\documentclass[reqno]{amsart}
\usepackage{amsmath}
\usepackage[dvips]{graphicx}
\usepackage{amsfonts}
\usepackage{amssymb}
\usepackage{latexsym}
\graphicspath{{Img/}}

\newtheorem{theorem}{Theorem}

\theoremstyle{plain}

\newtheorem{corollary}{Corollary}

\newtheorem{problem}{Problem}
\newtheorem{proposition}{Proposition}
\newtheorem{remark}{Remark}

\numberwithin{equation}{section}

\begin{document}
\title[Some minimum networks for four points]{Some minimum networks for four points in the three dimensional Euclidean Space}
\author{Anastasios Zachos}

\address{Pafou 17, 10446, Athens, Greece}
\email{azachos@gmail.com} 
\begin{abstract}
We construct a minimum tree for some boundary symmetric tetrahedra
$\mathbb{R}^{3},$ which has two nodes (interior points) with equal weights (positive numbers) having the property that the common perpendicular of some two opposite edges passes through their midpoints.  We prove that the length of this minimum tree may have length less than the length of the full Steiner tree for the same boundary symmetric tetrahedra.
\end{abstract}\maketitle

\section{Introduction}

In 1951, Courant and Robbins introduced the Steiner
problem (\cite[pp.360]{Cour/Rob:51}) in $\mathbb{R}^{2}:$

\begin{problem}{\cite[pp.360]{Cour/Rob:51}}
Given n points $A_{1},\cdots A_{n}$ to find a connected system of
straight line segments of shortest total length such that any two
of the given points can be joined by a polygon consisting of
segments of the system.
\end{problem}

The solution of the unweighted Steiner problem is called a Steiner
tree (\cite{GilbertPollak:68},\cite{Ci}).

A characterization of the solutions of the unweighted Steiner
problem in $\mathbb{R}^{3}$ is given by the following theorem:

\begin{theorem}{\cite[pp.~328]{BolMa/So:99},\cite{GilbertPollak:68}}
Any solution of the unweighted Steiner problem is a tree (a
Steiner tree) with at most $n-2$ Fermat-Torricelli points, where
each Fermat-Torricelli point has degree three and the angles
formed between any two edges incident with a Fermat-Torricelli
point are equal ($120^{\circ}$). The $n-2$ Fermat-Torricelli
points are vertices of the polygonal tree which do not belong to
$\{A_{1},\cdots,A_{n}\}.$
\end{theorem}

In 2011, we characterize a weighted Steiner minimal tree for convex quadrilaterals on the K-plane
(Two dimensional sphere with radius $R=\frac{1}{\sqrt{K}}, K>0,$ Hyperbolic plane with constant Gaussian curvature $K<0,$ Euclidean plane) in \cite[Theorem~2.1,pp.~140-149]{Zach:11}.

\begin{theorem}{\cite[Theorem~2.1,p.~140]{Zach:11}}
A weighted (full) Steiner minimal tree of $A_{1}A_{2}A_{3}A_{4}$
consists of two (weighted) Fermat-Torricelli points $A_{0},$
$A_{0}^{\prime}$ which are located at the interior convex domain
with corresponding weights $B_{0}$=$B_{0^{\prime}}$=$B_{5}$ and
minimizes the objective function:
\begin{equation} \label{eq:B_1}
B_1a_1+B_2a_2+B_3a_3+B_4a_4+B_{5}d=minimum,
\end{equation}
such that:
\begin{equation}\label{ineq1}
|B_{i}-B_{j}|<B_{k}<B_{i}+B_{j}
\end{equation}
and
\begin{equation}\label{ineq2}
|B_{l}-B_{m}|<B_{n}<B_{l}+B_{m}
\end{equation}
for $i,j,k\in\{1,4,5\},$ $l,m,n\in\{2,3,5\}$ and $i\ne j\ne k,$
$l\ne m\ne n.$
\end{theorem}

By setting $B_{4}=0,$ we obtain a weighted Fermat-Torricelli tree which coincides with the weighted Fermat-Torricelli tree w.r to the triangle $\triangle A_{1}A_{2}A_{3}.$

In this paper, we construct a minimum tree for some boundary symmetric tetrahedra in $\mathbb{R}^{3},$
whose common perpedicular of some two non-neighbouring edges pass from their midpoints.
By performing a rotation by a specific angle (twist angle) w.r to the line which pass from their midpoints the problem of finding a minimum network for these boundary symmetric tetrahedra is transformed to the equivalent problem of finding a minimum network for isosceles trapezium.

Thus, we consider the problem:

Find a minimum network which has two interior points  with corresponding weights (positive real numbers) which depend on the angle $\theta$ formed by the two diagonals of the isosceles trapezium (Problem~\ref{Steinertetrahedron}).

We shall solve Problem~\ref{Steinertetrahedron} by constructing two points which lie on the midperpendicular, which are the intersections of the distances of each vertex with the diagonals and we prove that the corresponding weights are $w(\theta)=\sin\frac{\theta}{2}$ (Solution of Problem~2).

Furthermore, we prove that the length of this minimum tree is less than the corresponding length of the Steiner tree
if $0^{\circ}<\theta<60^{o}$ (Theorem~3).

By applying Theorem~3 , we derive that the length of the construction tree may be less than the length of the Steiner tree for a rectangle (Corollary~1).
Finally, by taking into account Corollary~2 and remark~3, we show that the length of the construction tree is greater than the length of the corresponding Steiner tree for the square.


\section{The Steiner problem for some boundary symmetric tetrahedra in the three-dimensional Euclidean Space.}

We shall introduce the Steiner problem for some boundary symmetric tetrahedra in $\mathbb{R}^{3}.$
These boundary symmetric tetrahedra are tetrahedra whose common perpendicular of some two non-neighbouring edges  pass from their midpoints.

Let $A_{1}A_{2}A_{3}A_{4}$ be a tetrahedron in $\mathbb{R}^{3},$ such that
$d$ is the length of the common perpendicular of the edges $A_{1}A_{2}$ and $A_{4}A_{3}$ (euclidean distance) which pass from the midpoints $M_{12}$ and $M_{34}$ of $A_{1}A_{2}$ and $A_{4}A_{3},$ respectively.

We denote by $F_{12},$ $F_{34}$ two points at the interior of $A_{1}A_{2}A_{3}A_{4}$ in $\mathbb{R}^{3}$
with corresponding positive numbers (weights) $w_{12}$ and $w_{34},$ respectively,
by $a_{i,12}$ the Euclidean distance of the line segment $A_{i}F_{12},$
by $a_{i,34}$ the Euclidean distance of the line segment $A_{i}F_{34},$
$a_{ij}$ the Euclidean distance of the line segment $A_{i}A_{j},$ for
$i,j=1,2,3,4$ and by $d_{12,34}$ the the Euclidean distance of the line segment $F_{12}F_{34}.$

The twist angle is referred as the angle between the planes formed by $\triangle A_{1}A_{2}F_{12} $ and $\triangle A_{4}A_{3}F_{34},$ at the edge $F_{12}F_{34}.$

The twist angle $\varphi$ for this particular tetrahedron $A_{1}A_{2}A_{3}A_{4}$ is given by:

\begin{equation}\label{twist}
\varphi=\arccos(\frac{\vec{a_{12}}\vec{a_{43}}}{a_{12}a_{43}}).
\end{equation}
By rotating $A_{1}A_{2}$ w.r. to $M_{12}$ by an angle $\varphi$
we derive an isosceles trapezium $A_{1}^{\prime}A_{2}^{\prime}A_{3}A_{4}.$

We denote by $F$ the intersection point of the two equal diagonals $A_{1}^{\prime}A_{3}$ and $A_{2}^{\prime}A_{4}$ and by $\theta$ the angle $\angle A_{1}^{\prime}FA_{2}^{\prime}=\angle A_{4}FA_{3}.$

Assume that $d>\max \{a_{12},a_{34}\}.$
\begin{problem}\label{Steinertetrahedron}
Find $F_{12}$ and $F_{34}$ with corresponding weights (positive real numbers) $w_{12}$ and $w_{34},$ such that
\[w_{12}=w_{34}=w(\theta)>0\]
and
\begin{equation}\label{equat1}
f(a_{1,12},a_{2,12},a_{3,34},a_{4,34},\theta, d)=a_{1,12}+a_{2,12}+a_{3,34}+a_{4,34}+w(\theta) d_{12,34}\to min
\end{equation}

\end{problem}

\begin{proof}[Solution of Problem~\ref{Steinertetrahedron}]
Without loss of generality, we assume that:

$M_{43}=\{0,0,0\},$ $M_{34}M_{12}$ lie on the $z$ axis, $M_{12}=\{0,0,z_{1}\},$
$A_{1}=\{-x_{1},-y_{1},z_{1}\},$ $A_{2}=\{x_{1},y_{1},z_{1}\},$ $A_{4}=\{-x_{4},0,0\},$ $A_{3}=\{x_{4},0,0\}.$

The angle $\varphi$ is given by:

\begin{equation}\label{twistcalc}
\varphi=\arccos(\frac{\{2x_{4},0,0\}\{2x_{1},2y_{1},0\}}{4x_{4}\sqrt{(x_{1}^2+y_{1}^2)})}.
\end{equation}

or

\begin{equation}\label{twistcalc2}
\varphi=\arccos(\frac{x_{1}}{\sqrt{(x_{1}^2+y_{1}^2)}}).
\end{equation}

By rotating by $\varphi$ counterclockwise $A_{1}A_{2}$ w.r. to $M_{12},$
we derive an isosceles trapezium $A_{1}^{\prime}A_{2}^{\prime}A_{4}A_{3}.$
We get:
$x_{1}^{\prime}=-\frac{x_{1}}{\cos\varphi},$ $y_{1}^{\prime}=0,$ $z_{1}^{\prime}=z_{1},$
$x_{2}^{\prime}=\frac{x_{1}}{\cos\varphi},$ $y_{2}^{\prime}=0,$ $z_{2}^{\prime}=z_{1}.$

From $\triangle{A_{1}^{\prime}FM_{12}}$ and $\triangle{A_{3}FM_{12}},$ we derive that:

\begin{equation}\label{anglediagon1}
\tan{\theta}=\frac{A_{3}M_{34}}{d-FM_{12}}
\end{equation}

and

\begin{equation}\label{anglediagon2}
\tan{\theta}=\frac{A_{1}M_{12}}{FM_{12}}.
\end{equation}

By subtracting (\ref{anglediagon1}) from (\ref{anglediagon2}), we obtain:

\begin{equation}\label{anglediagonals}
FM_{12}=\frac{d}{\frac{A_{3}M_{34}}{A_{1}M_{12}}+1}
\end{equation}

or

\begin{equation}\label{anglediagonals2}
FM_{12}=\frac{d}{\frac{a_{34}}{a_{12}}+1}
\end{equation}

By replacing (\ref{anglediagonals2}) in (\ref{anglediagon2}), we get:

\begin{equation}\label{anglediagoncalc}
\tan{\theta}=\frac{a_{12}}{\frac{2d}{\frac{a_{34}}{a_{12}}+1}}.
\end{equation}

\begin{figure}\label{isosctrapezium1}
\centering
\includegraphics[scale=0.80]{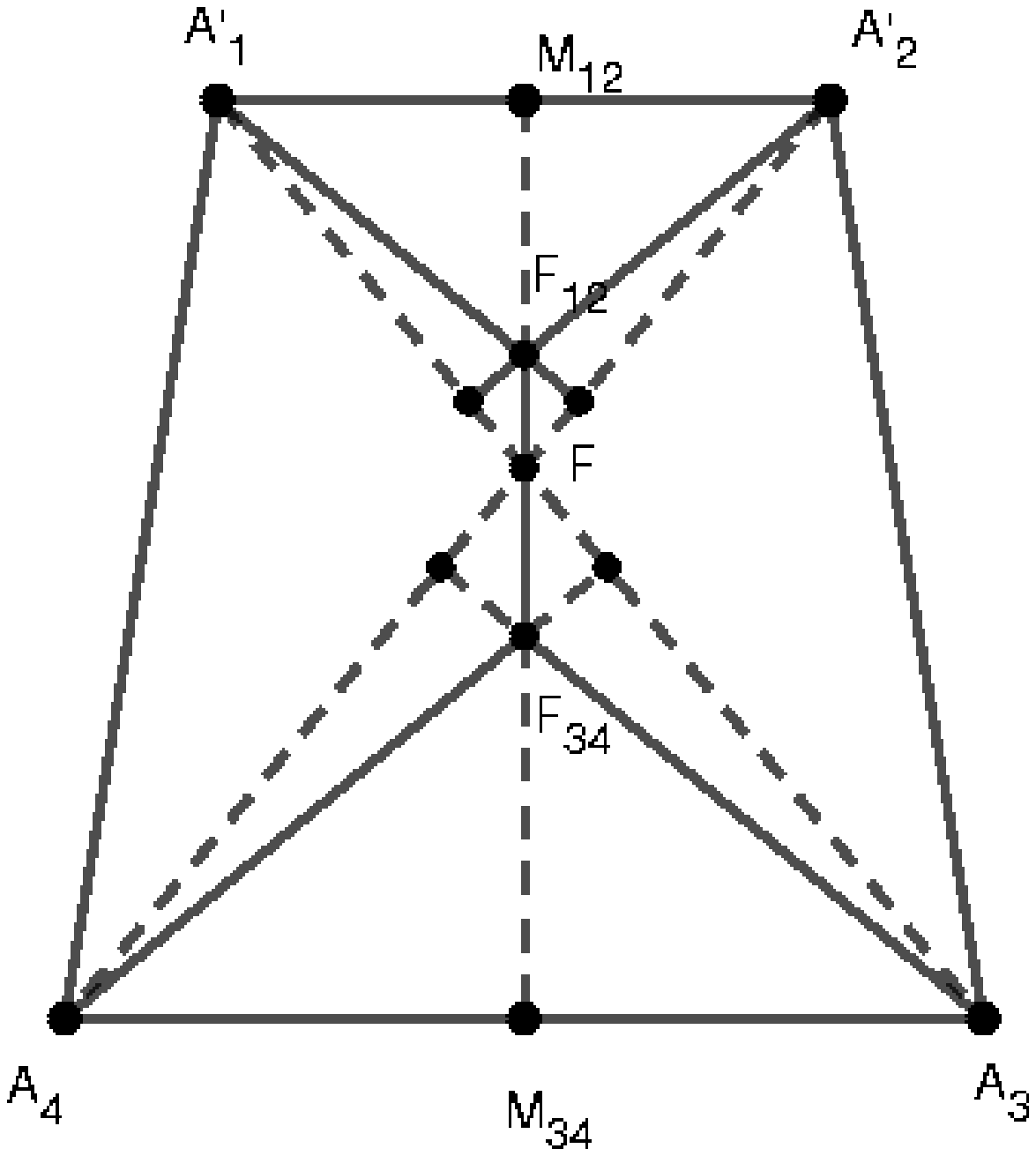}
\caption{}
\end{figure}

The intersection of the two heights of $\triangle A_{3}FA_{4}$ w.r to the sides $FA_{3}$ and $FA_{4}$
is the point $F_{34}.$

The intersection of the two heights of $\triangle A_{1}^{\prime}FA_{2}^{\prime}$ w.r to the sides $FA_{1}^{\prime}$ and $FA_{2}^{\prime}$ is the point $F_{12}.$
The points $F_{12},$ $F_{34}$ belong to $M_{12}M_{34}.$

Thus, we get

\begin{equation}\label{theta1}
\angle A_{3}F_{34}A_{4}=\angle A_{1}^{\prime}F_{12}A_{2}^{\prime}=180^{\circ}-\theta.
\end{equation}

By applying Theorem~2, for $A_{1}^{\prime}A_{2}^{\prime}A_{4}A_{3},$
$B_{1}=B_{2}=B_{3}=B_{4}=1,$ we derive that $B_{0}=B_{0^{\prime}}=w(\theta).$

Taking into account that $F_{12}$ is the weighted Fermat-Torricelli point of $\triangle A_{1}^{\prime}F_{34}A_{2}^{\prime}$ and $F_{34}$ is the weighted Fermat-Torricelli point of $\triangle A_{4}F_{12}A_{3},$
we derive that:
\begin{equation}\label{calcweighttheta1}
\frac{1}{\sin(90^{\circ}+\frac{\theta}{2})}=\frac{w(\theta)}{\sin(180^{\circ}-\theta)},
\end{equation}

or

\begin{equation}\label{calcweighttheta2}
w(\theta)=2\sin\frac{\theta}{2}.
\end{equation}

\end{proof}

We denote by $O_{12}$ and $O_{34}$ the corresponding Fermat-Torricelli points of $\triangle A_{4}O_{34}A_{3}$ and $\triangle A_{1}^{\prime}O_{12}A_{2}^{\prime}$ which lie on $M_{12}M_{34},$ where
$\angle A_{1}^{\prime}O_{12}A_{2}=\angle A_{1}^{\prime}O_{12}O_{34}=\angle A_{2}^{\prime}O_{12}O_{34}=120^{\circ},$ and
$\angle A_{4}O_{34}A_{3}=\angle A_{4}O_{34}O_{12}=\angle A_{3}O_{34}O_{12}=120^{\circ}.$

We set
\begin{equation}\label{lengthST1}
 l_{minST}(A_{1}^{\prime}A_{2}^{\prime}A_{4}A_{3})\equiv 2 A_{1}^{\prime}O_{12}+2 A_{3}O_{34}+O_{12}O_{34}
\end{equation}

and

\begin{equation}\label{lengthT1}
l_{minT}(A_{1}^{\prime}A_{2}^{\prime}A_{4}A_{3})\equiv 2 A_{1}^{\prime}F_{12}+2 A_{3}F_{34}+w(\theta) F_{12}F_{34}.
\end{equation}

\begin{proposition}\label{lengthisosctrapConstructiontree}
The length of the minimum construction tree of $A_{1}^{\prime}A_{2}^{\prime}A_{4}A_{3}$
having two weighted Fermat-Torricelli points $F_{12}$ and $F_{34}$ with corresponding equal weights $w(\theta)=2\sin\frac{\theta}{2}$ is given by
\begin{equation}\label{lengthT2}
l_{minT}(A_{1}^{\prime}A_{2}^{\prime}A_{4}A_{3})=2(a_{34}+a_{12})\cos\frac{\theta}{2}.
\end{equation}

\end{proposition}

\begin{proof}
From $\triangle A_{3}F_{34}M_{34},$ and $\triangle A_{1}^{\prime}F_{12}M_{12},$ we derive:
\begin{equation}\label{length334}
A_{3}F_{34}=\frac{a_{34}}{2 \cos\frac{\theta}{2}},
\end{equation}

\begin{equation}\label{length112}
A_{1}^{\prime}F_{12}=\frac{a_{12}}{2 \cos\frac{\theta}{2}},
\end{equation}

\begin{equation}\label{length3434}
F_{34}M_{34}=\frac{a_{34}}{2} \tan\frac{\theta}{2},
\end{equation}

and

\begin{equation}\label{length1212}
F_{12}M_{12}=\frac{a_{12}}{2} \tan\frac{\theta}{2}.
\end{equation}

Taking into account that

\begin{equation}\label{length1234}
F_{12}F_{34}=d-F_{12}M_{12}-F_{34}M_{34}
\end{equation}

and by replacing (\ref{length3434}) and (\ref{length1212}) in (\ref{length1234}),we
get:

\begin{equation}\label{length1234theta}
F_{12}F_{34}=d-\frac{(a_{12}+a_{34})}{2} \tan\frac{\theta}{2}.
\end{equation}

Taking into account that

\begin{equation}\label{lengthm1234theta}
d=M_{12}M_{34}=(a_{12}+a_{34})\frac{1}{ 2\tan\frac{\theta}{2}}.
\end{equation}

and by replacing (\ref{length1234theta}), (\ref{lengthm1234theta}),
(\ref{length334}) and (\ref{length112}) in (\ref{lengthT1}) we obtain
(\ref{lengthT2}).

\end{proof}

\begin{proposition}\label{lengthisosctrapSteinertree}
The length of the full (equally weighted) Steiner tree of $A_{1}^{\prime}A_{2}^{\prime}A_{4}A_{3}$
is given by
\begin{equation}\label{lengthST2}
l_{minST}(A_{1}^{\prime}A_{2}^{\prime}A_{4}A_{3})=(a_{34}+a_{12})(\frac{\sqrt{3}}{2}+\frac{\cos\frac{\theta}{2}}{2\sin\frac{\theta}{2}})
\end{equation}

\end{proposition}

\begin{proof}

From $\triangle A_{3}O_{34}M_{34},$ and $\triangle A_{1}^{\prime}O_{12}M_{12},$ we derive:
\begin{equation}\label{length334st}
A_{3}O_{34}=\frac{a_{34}}{2 \cos30^{\circ}},
\end{equation}

\begin{equation}\label{length112st}
A_{1}^{\prime}O_{12}=\frac{a_{12}}{2 \cos30^{\circ}},
\end{equation}

\begin{equation}\label{length3434st}
O_{34}M_{34}=\frac{a_{34}}{2} \tan30^{\circ},
\end{equation}

and

\begin{equation}\label{length1212st}
O_{12}M_{12}=\frac{a_{12}}{2} \tan30^{\circ}.
\end{equation}

Taking into account that

\begin{equation}\label{length1234st}
O_{12}O_{34}=d-O_{12}M_{12}-O_{34}M_{34}
\end{equation}

and by replacing (\ref{length3434st}) and (\ref{length1212st}) in (\ref{length1234st}),we
get:

\begin{equation}\label{length1234stst}
O_{12}O_{34}=d-\frac{(a_{12}+a_{34})}{2} \tan30^{\circ}.
\end{equation}

By replacing (\ref{length1234stst}), (\ref{lengthm1234theta}),
(\ref{length334st}) and (\ref{length112st}) in (\ref{lengthST1}) we obtain
(\ref{lengthST2}).

\end{proof}

\begin{figure}\label{calcanglethetagraph}
\centering
\includegraphics[scale=0.80]{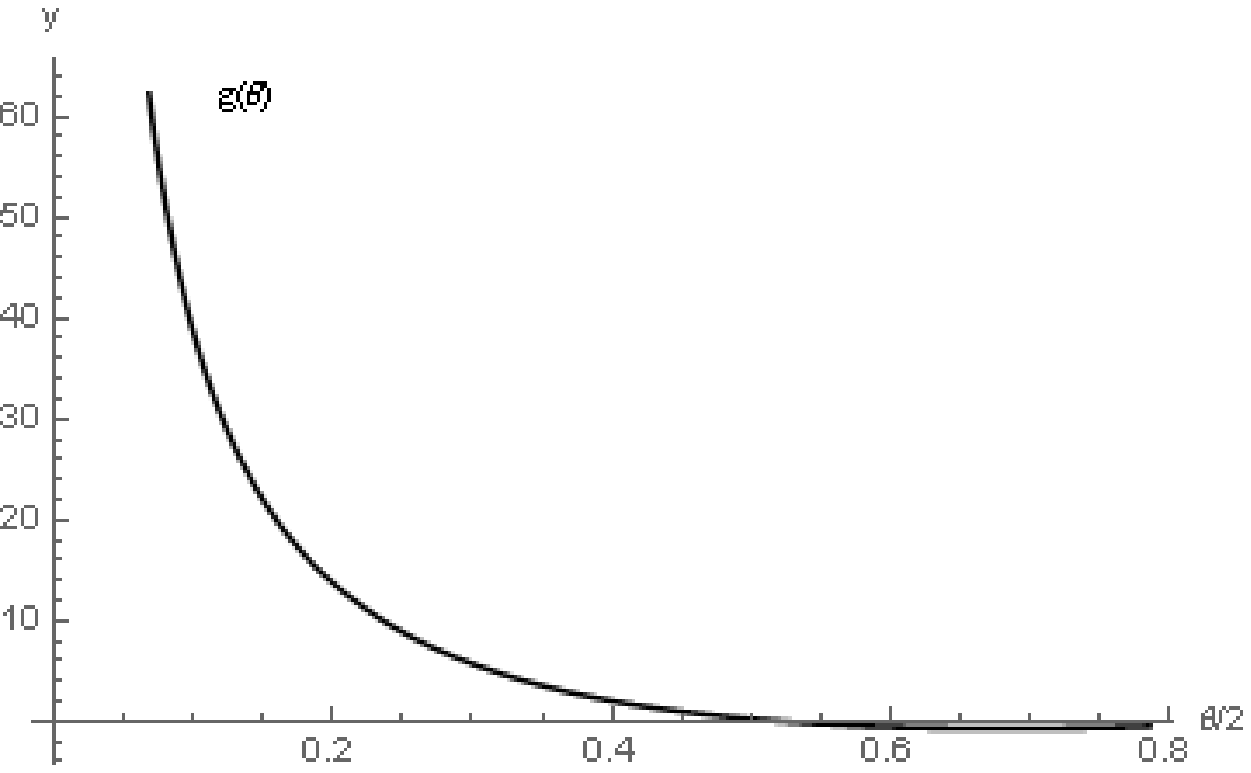}
\caption{}
\end{figure}

We consider a class of isosceles trapezium $A_{1}^{\prime}A_{2}^{\prime}A_{4}A_{3},$
such that $a_{12}$, $a_{34}$ are constant positive real numbers and $\theta=\angle A_{1}^{\prime}FA_{2}^{\prime},$ $d=M_{12}M_{34}$ are
variables. The class of isosceles trapezium are the isosceles trapezium which are formed by a parallel translation of $a_{12}$ or $a_{34}$ w.r. to $M_{12}M_{34}.$
\begin{theorem}\label{minimumtree}
If $0<\theta<60^{\circ},$ then \[l_{minT}(A_{1}^{\prime}A_{2}^{\prime}A_{4}A_{3})<l_{minST}(A_{1}^{\prime}A_{2}^{\prime}A_{4}A_{3}),\]

and If $60^{\circ}<\theta<90^{\circ},$ then \[l_{minT}(A_{1}^{\prime}A_{2}^{\prime}A_{4}A_{3})>l_{minST}(A_{1}^{\prime}A_{2}^{\prime}A_{4}A_{3}).\]

\end{theorem}

\begin{proof}
We set
\begin{equation}\label{difference1}
g(\theta)=l_{minST}(A_{1}^{\prime}A_{2}^{\prime}A_{4}A_{3})-l_{minT}(A_{1}^{\prime}A_{2}^{\prime}A_{4}A_{3})
\end{equation}

or

\begin{equation}\label{difference2}
g(\theta)=(a_{12}+a_{12})(\frac{\sqrt{3}}{2}+\frac{\cos\frac{\theta}{2}}{2\sin\frac{\theta}{2}}-2\cos\frac{\theta}{2}).
\end{equation}

By replacing the trigonometric transformations

\begin{equation}\label{tang1}
\tan{\frac{\theta}{2}}=\frac{2\tan^{2}\frac{\theta}{4}}{1-\tan^{2}\frac{\theta}{4}}
\end{equation}

and

\begin{equation}\label{tang2}
\cos{\frac{\theta}{2}}=\frac{1-\tan^{2}\frac{\theta}{4}}{1+\tan^{2}\frac{\theta}{4}}
\end{equation}

in (\ref{difference2})
and by setting \[t=\tan^{2}\frac{\theta}{4},\] we obtain a polynomial of fourth order w.r. to $t:$

\begin{equation}\label{solvet}
-t^4+(2\sqrt{3}+8) t^3+(2 \sqrt{3}-8) t+1=0,
\end{equation}
There is only one real solution $t=0.26794919243112275,$ which gives $\theta =60^{\circ}\in (0^{\circ},90^{\circ}).$
The other real solutions give values for $\theta\notin (0^{\circ},90^{\circ}).$
The function attains a global minimum at $\theta\approx 78.09^{\circ}$
which corresponds to the Fermat condition of the first derivative
$\sin^{3}\frac{\theta}{2}=\frac{1}{4}.$

Thus, we derive that $g(\theta)$ is decreasing for $\theta \in (0^{\circ},78.09^{\circ})$

increasing  for $\theta \in (78.09^{\circ},90^{\circ})$
and positive for $\theta \in (0^{\circ},60^{\circ})$
(Fig~\ref{calcanglethetagraph}).

\end{proof}

\begin{corollary}\label{rectangleex1}
If $a_{12}=a_{34}$ and $\theta<90^{\circ},$ then
\begin{equation}\label{lengthT2rectangle}
l_{minT}(A_{1}^{\prime}A_{2}^{\prime}A_{4}A_{3})=4a_{12}\cos\frac{\theta}{2}.
\end{equation}

and

\begin{equation}\label{lengthST2rectangle}
l_{minST}(A_{1}^{\prime}A_{2}^{\prime}A_{4}A_{3})=(a_{12})\sqrt{3}+\frac{\cos\frac{\theta}{2}}{2\sin\frac{\theta}{2}})
\end{equation}

\end{corollary}

\begin{proof}
For $a_{12}=a_{34},$ $A_{1}^{\prime}A_{2}^{\prime}A_{4}A_{3}$ is a rectangle.

By replacing $a_{12}=a_{34}$ in (\ref{lengthT2}) and (\ref{lengthST2}), we obtain

(\ref{lengthT2rectangle}) and (\ref{lengthST2rectangle}).
\end{proof}

\begin{remark}
By replacing $\cos\frac{\theta}{2}=\frac{d}{\sqrt{d^2+a_{12}^2}}$ in (\ref{lengthST2rectangle}) we get
\begin{equation}\label{lengthT2rectanglealt}
l_{minT}(A_{1}^{\prime}A_{2}^{\prime}A_{4}A_{3})=4a_{12}\frac{d}{\sqrt{d^2+a_{12}^2}}.
\end{equation}
\end{remark}

\begin{corollary}\label{square}
If $a_{12}=a_{34}$ and $\theta=90^{\circ},$ then
\begin{equation}\label{lengthT2square}
l_{minT}(A_{1}^{\prime}A_{2}^{\prime}A_{4}A_{3})=2a_{12}\sqrt{2},
\end{equation}

\begin{equation}\label{lengthST2square}
l_{minST}(A_{1}^{\prime}A_{2}^{\prime}A_{4}A_{3})=a_{12}(\sqrt{3}+1)
\end{equation}

and

\[l_{minT}(A_{1}^{\prime}A_{2}^{\prime}A_{4}A_{3})>l_{minST}(A_{1}^{\prime}A_{2}^{\prime}A_{4}A_{3}).\]

\end{corollary}

\begin{proof}
For $a_{12}=a_{34}$ and $\theta=90^{\circ},$ $A_{1}^{\prime}A_{2}^{\prime}A_{4}A_{3}$ is a square.

By replacing $a_{12}=a_{34}$ and $\theta=90^{\circ}$ in (\ref{lengthT2}) and (\ref{lengthST2}), we obtain

(\ref{lengthT2square}) and (\ref{lengthST2square}).

The inequality \[\sqrt{3}+1< 2\sqrt{2}\] yields
\[l_{minT}(A_{1}^{\prime}A_{2}^{\prime}A_{4}A_{3})>l_{minST}(A_{1}^{\prime}A_{2}^{\prime}A_{4}A_{3}).\]

\end{proof}

\begin{remark}
Taking into account Corollary~\ref{square}, the length of the minimum tree $T$ for the square $A_{1}^{\prime}A_{2}^{\prime}A_{4}A_{3}$ is the sum of the two equal diagonals. The intersection $F$ of the two diagonals is the Fermat-Torricelli point of the square $A_{1}^{\prime}A_{2}^{\prime}A_{4}A_{3}.$
\end{remark}

\begin{remark}
The length of the minimum tree $T$ for the rectangle $A_{1}^{\prime}A_{2}^{\prime}A_{4}A_{3}$ is given by:

\[l_{minT}(A_{1}^{\prime}A_{2}^{\prime}A_{4}A_{3})=4\frac{a12}{\cos\frac{\theta}{2}}+2\cos\frac{\theta}{2}(d-\frac{a_{12}^2}{d}).\]

For $a_{12}=d,$  $A_{1}^{\prime}A_{2}^{\prime}A_{4}A_{3}$ is a square and the second term vanishes which corresponds to the weight $w(\theta).$
It is important to note that for this reason the length of the Steiner minimum tree of the square is less than the length of the minimum construction tree $T$ for the same square (Fermat-Torricelli tree).
\end{remark}


\begin{thebibliography}{99}
\bibitem{BolMa/So:99} V. Boltyanski, H. Martini, V. Soltan, \emph{Geometric Methods and Optimization Problems}, Kluwer, Dordrecht-Boston-London, 1999.
\bibitem{Ci} D. Cieslik, \emph{Steiner minimal trees}. Nonconvex Optimization and its Applications, 23.
Kluwer Academic Publishers, Dordrecht, 1998.
\bibitem{Cour/Rob:51} R. Courant and H. Robbins, \emph{What is Mathematics?} Oxford University Press., New York, 1951.
\bibitem{GilbertPollak:68} E.N. Gilbert and H.O. Pollak, \emph{Steiner Minimal
trees}, SIAM Journal on Applied Mathematics.\textbf{16} (1968),
1-29.
\bibitem{RubinsteinThomasWeng:02} J. H. Rubinstein, D. A. Thomas and J. Weng, \emph{Minimum networks for four Points in Space}, Geom. Dedicata. \textbf{93} (2002), 57-70.

\bibitem{WengThomasMareels:09} J. F. Weng, D. A. Thomas and I. Mareels, \emph{Identifying Steiner minimal trees on four points in Space}, Discrete Math. Algorithms Appl. \textbf{1} (3) (2009), 401-411.


\bibitem{Zach:11} A. Zachos , A weighted Steiner minimal tree for convex quadrilaterals on the two dimensional K-plane, \emph{J. Convex. Anal.} \textbf{19} (1), (2011),139-152.
\end{thebibliography}
\end{document}